\documentclass[twoside,12pt]{article}
\usepackage{graphicx}
\usepackage{epsfig}
\usepackage{amssymb,amsthm,amsmath,amsfonts,amstext,mathrsfs}
\usepackage{latexsym}
\usepackage{url}
\newtheorem{thm}{Theorem}

\newtheorem{cor}{Corollary}

\newcommand{\cao}{{\c c}{\~a}o}

\input xy
\xyoption{all}

\def\d{\,{\rm{d}}}

\begin{document}

\title{\bf \Large{On some Rajchman measures and equivalent Salem's problem}}
\author{Semyon Yakubovich}
\maketitle

 \markboth{\rm \centerline{SEMYON YAKUBOVICH}}{}
 \markright{\rm \centerline{RAJCHMAN\  MEASURES}}

 \begin{abstract}{We construct certain  Rajchman measures
 by using integrability properties of the Fourier and
Fourier-Stieltjes transforms. In particular,  we state a problem and prove
that it is equivalent to the known and still unsolved question posed
by R. Salem ({\it Trans. Amer. Math. Soc.} {\bf 53} (3) (1943),  p.
439) whether Fourier-Stieltjes coefficients of the Minkowski's
question mark function vanish at infinity.}
\end{abstract}

{\bf Keywords}:  Rajchman measure, Minkowski question mark function,
Salem's problem, Fourier-Stieltjes transform, modified Bessel
function, index integral, Fourier-Stieltjes coefficients

{\bf Mathematics  subject classification}:  33C10, 42A16, 42B10,
44A15

\section{Introduction and auxiliary results}

It is well known in the elementary theory of the Fourier-Stieltjes
integrals that if $h(x)$ is absolutely continuous then
\begin{eqnarray}\label{F}
g(t)=\int\limits_{\Omega} e^{ ixt} dh(x), \ \Omega \subset
\mathbb{R}, \ t \in \mathbb{R}
\end{eqnarray}
tends to zero as $|t| \to \infty,$  because in this case the
Fourier-Stieltjes transform $g(t)$ is an ordinary Fourier transform
of an integrable function. Thus $h(x)$ supports a measure whose
Fourier transform vanishes at infinity. Such measures are called
Rajchman measures (see details, for instance, in \cite{LYO}).
However, when $h$ is continuous, the situation is quite different
and the classical Riemann-Lebesgue lemma for the class $L_1$, in
general, cannot be applied. The question is quite delicate when it
concerns singular monotone functions (see \cite{Sal3}, Ch. IV). For
such singular measures there are various examples and the
Fourier-Stieltjes transform need not tend to zero,  although there
do exist measures for which it goes to zero. For instance, Salem
\cite{Sal1}  \cite{Sal2} gave examples of singular functions, which
are strictly increasing and whose Fourier coefficients still do not
vanish at infinity. On the other hand, Menchoff in 1916 \cite{Men}
gave a first example of a singular distribution whose coefficients
vanish at infinity. Wiener and Wintner (see also \cite{IV}) proved
in 1938 that for every $\varepsilon >0$ there exists a singular
monotone function such that its Fourier coefficients behave as
$n^{-{1\over 2} +\varepsilon}, \ n \to \infty$.

Our goal in this paper is to construct  some  Raijchman's measures,
which are associated with  continuous functions of bounded
variation. In particular, we will prove  that the famous Minkowski'
s question mark function $?(x)$ \cite{D} is a Raijchman measure if
and only if its Forier-Stieltjes transform has a limit at infinity,
and then, of course,  the limit should be zero. This probably can
give an affirmative answer on the question posed by Salem in 1943 \cite{Sal1}.

The Minkowski question mark function $?(x):[0,1]\mapsto[0,1]$ is
defined by \cite{D}
\begin{equation}
?([0,a_{1},a_{2},a_{3},\ldots])=2\sum\limits_{i=1}^{\infty}(-1)^{i+1}2^{-\sum_{j=1}^{i}a_{j}},
\label{min}
\end{equation}
where $x=[0,a_{1},a_{2},a_{3},\ldots]$ stands for the representation
of $x$ by a regular continued fraction. We will keep the notation
$?(x)$, which was used  in  the original Salem's paper \cite{Sal1},
mildly resisting the temptation of changing it and despite this
symbol is quite odd to denote a function in such a way.  It is well
known that $?(x)$ is continuous, strictly increasing and singular
with respect to Lebesgue measure. It can be extended on $[0,
\infty]$ by using the following  functional equations
\begin{equation}\label{rel1}
?(x)= 1-?(1-x),
\end{equation}
\begin{equation}\label{rel2}
?(x)= 2?\left(\frac{x}{x+1}\right),
\end{equation}
\begin{equation}\label{rel3}
?(x)+ ?\left({1\over x}\right)= 2, \  x > 0.
\end{equation}
When $x \to 0$, it decreases exponentially $ \ ?(x)= O\left(2^{-
1/x}\right)$. Key values are $?(0)=0, \ ?(1)= 1,\ ?(\infty)= 2$. For
instance, from \eqref{rel1} and asymptotic behavior of the Minkowski
function $?(x)$ near zero one can easily get  the finiteness of the
following integrals
\begin{equation}\label{0}
\int_0^1 x^\lambda\  d ?(x) < \infty, \ \lambda \in \mathbb{R},
\end{equation}
\begin{equation}\label{1}
\int_0^1 (1-x)^\lambda\ d ?(x) < \infty, \ \lambda \in \mathbb{R}.
\end{equation}
Further, as was proved by Salem \cite{Sal1}, the Minkowski question
mark function satisfies the H$\ddot o$lder condition
\begin{equation}\label{Lip}
\left|?(x)- ?(y)\right| < C |x-y|^\alpha,
\end{equation}
of order
\begin{equation}\label{exp}
\alpha= \frac{\log 2}{2 \log {\sqrt 5 + 1\over 2}},
\end{equation}
where $C >0$ is an absolute constant. We will deal in the sequel
with the following Fourier -Stieltjes transforms of the Minkowski
question mark function
\begin{eqnarray}\label{fF}
f(t)=\int\limits_{0}^{1}e^{ixt}\ d ?(x),
 \quad  F(t)=\int\limits_{0}^{\infty}e^{ixt}\ d ?(x),\quad t \in \mathbb{R},
\end{eqnarray}
\begin{eqnarray}\label{fFc}
f_c(t)=\int\limits_{0}^{1}\cos xt \ d ?(x), \quad
F_c(t)=\int\limits_{0}^{\infty}\cos xt \ d ?(x),\ t \in
\mathbb{R_+},
\end{eqnarray}
\begin{eqnarray}\label{fFs} f_s(t)=\int\limits_{0}^{1}\sin xt \ d
?(x), \quad  F_s(t)=\int\limits_{0}^{\infty}\sin xt \d ?(x),\ t \in
\mathbb{R_+},
\end{eqnarray}
where all integrals converge absolutely and uniformly with respect
to $t$ because of straightforward estimates
\begin{eqnarray*}
|f(t)| \le \int\limits_{0}^{1} d ?(x) =1,
 \quad  |F(t)| \le \int\limits_{0}^{\infty}  d ?(x) = 2,
\end{eqnarray*}
\begin{eqnarray*}
|f_c(t)| \le 1, \quad |F_c(t)| \le 2,
\end{eqnarray*}
\begin{eqnarray*}
| f_s(t)| \le 1, \quad |F_s(t)| \le 2.
\end{eqnarray*}
 Further we observe that functional equation (\ref{rel1}) easily implies $f(t)= e^{it} f(-t)$ and
therefore $e^{-it/2} f(t) \in \mathbb{R}$. So, taking the imaginary
part we obtain the equality
\begin{eqnarray}\label{imprel1}
\cos \left({t\over 2}\right) f_s(t)=  \sin \left({t\over 2}\right)
f_c(t).
\end{eqnarray}
Hence, for instance,  letting $t= 2\pi  n, \ n \in \mathbb{N}_0$ it
gives $f_s(2\pi n)=0$ and $f_c(2\pi n)= d_n$.  In 1943  Salem asked
\cite{Sal1} whether $d_{n}\rightarrow 0$, as $n\rightarrow\infty$.

Further, by using functional equations (\ref{rel2}), (\ref{rel3})
for the Minkowski function we derive the following useful relations
\begin{eqnarray*}
\int\limits_{0}^{1}e^{ixt}\ d ?(x)= \int\limits_{0}^{\infty}e^{ixt}\
d ?(x)- \int\limits_{1}^{\infty}e^{ixt}\ d ?(x)\\
= \int\limits_{0}^{\infty}e^{ixt}\ d ?(x) + e^{it}
\int\limits_{0}^{\infty}e^{ixt}\
d ?\left({1\over x+1}\right)\\
= \int\limits_{0}^{\infty}e^{ixt}\ d ?(x) + {e^{it}\over 2}
\int\limits_{0}^{\infty}e^{ixt}\
d ?\left({1\over x}\right)\\
= \left(1- {e^{it}\over 2}\right) \int\limits_{0}^{\infty}e^{ixt}\ d
?(x),
\end{eqnarray*}
which imply the functional equation
\begin{eqnarray}\label{imprel2}
F(t)= \frac{2 f(t)}{2- e^{it}}.
\end{eqnarray}
Taking real and imaginary parts in (\ref{imprel2}) and employing
functional equation (\ref{rel1}) it is not difficult to deduce the
following important equalities for the Fourier-Stieltjes transforms
(\ref{fFc}), (\ref{fFs})
\begin{eqnarray}\label{imprel3}
F_c(t)= \frac{2 }{5- 4\cos t}\ f_c(t),
\end{eqnarray}
\begin{eqnarray}\label{imprel4}
F_s(t)= \frac{6}{5- 4\cos t}\ f_s(t).
\end{eqnarray}
Indeed, we have, for instance
\begin{eqnarray*}
\int\limits_{0}^{\infty}\cos xt \ d ?(x)= \frac{2 }{5- 4\cos
t}\left[(2-\cos t)\int\limits_{0}^{1}\cos xt \ d ?(x)- \sin
t\int\limits_{0}^{1}\sin xt \ d ?(x) \right]\\
 = \frac{2 }{5- 4\cos
t}\left[2 \int\limits_{0}^{1}\cos xt \ d ?(x)-
\int\limits_{0}^{1}\cos t(1-x)\ d ?(x) \right]\\
=  \frac{2 }{5- 4\cos t} \int\limits_{0}^{1}\cos xt \ d ?(x)
\end{eqnarray*}
and this yields relation (\ref{imprel3}). Analogously we get
(\ref{imprel4}). In particular, letting $t= 2\pi n, n \in
\mathbb{N}_0$ in (\ref{imprel3}), (\ref{imprel4}) we find
accordingly
\begin{eqnarray*}
\int\limits_{1}^{\infty}\cos (2\pi n x) \ d ?(x)=
\int\limits_{0}^{1}\cos (2\pi  n x)\ d ?(x),
\end{eqnarray*}
\begin{eqnarray*}
\int\limits_{1}^{\infty}\sin (2\pi n x) \ d ?(x)= 5
\int\limits_{0}^{1}\sin (2\pi  n x)\ d ?(x)=0
\end{eqnarray*}
via  (\ref{imprel1}). Generally, equalities (\ref{imprel3}),
(\ref{imprel4}) yield
\begin{eqnarray*}
\int\limits_{1}^{\infty}\cos xt \ d ?(x)= \frac{1- 8\sin^2(t/2)}{1+
8\sin^2(t/2)}\int\limits_{0}^{1}\cos x t\ d ?(x),
\end{eqnarray*}
\begin{eqnarray*}
\int\limits_{1}^{\infty}\sin xt \ d ?(x)= \frac{5- 8\sin^2(t/2)}{1+
8\sin^2(t/2)}\int\limits_{0}^{1}\sin  xt\ d ?(x).
\end{eqnarray*}
respectively. For instance,
\begin{eqnarray*}
\int\limits_{1}^{\infty}\cos (xt_m) \ d ?(x)= 0,
\end{eqnarray*}
\begin{eqnarray*}
\int\limits_{1}^{\infty}\sin (xt_k) \ d ?(x)= 0
\end{eqnarray*}
for any $t_m, \ t_k$, which are roots of the corresponding equations
$$\sin(t_m/2)= \pm 1/(2\sqrt 2), \ \sin(t_k/2)= \pm \sqrt{5/8}, \quad
m, k \in \mathbb{N}.$$

 Further, since (see (\ref{imprel2}),\
(\ref{imprel3}), \
 (\ref{imprel4}))
\begin{eqnarray}\label{inF}
{1\over 2}\ |F(t)| \le |f(t)| \le  {3\over 2} |F(t)|,
\end{eqnarray}
\begin{eqnarray}\label{inFc}
{1\over 2} \ |F_c(t)| \le |f_c(t)| \le  {9\over 2} |F_c(t)|,
\end{eqnarray}
\begin{eqnarray}\label{inFs}
{1\over 6}\  |F_s(t)| \le  |f_s(t)| \le {3\over 2} |F_s(t)|,
\end{eqnarray}
then Fourier-Stieltjes transforms of the Minkowski question mark
function over $(0,1)$ tend to zero when $|t| \to \infty$ if and only
if the same property is guaranteed by Fourier-Stieltjes transforms
over $(0,\infty)$.

We will show in the next section that the finite Fourier-Stieltjes
transform can be treated with the use of the so-called
Lebedev-Stieltjes integrals, involving the modified Bessel function
$K_{i\tau}(x)$ of the pure imaginary index \cite{YAK1}.  It is known
that the modified Bessel function $K_\mu(z)$ satisfies the
differential equation
\begin{equation*}
 z^2{d^2u\over dz^2}  + z{du\over dz} - (z^2+\mu^2)u = 0
\end{equation*}
and  has the following asymptotic behavior
\begin{equation}\label{as1}
K_\mu(z) = \left( \frac{\pi}{2z} \right)^{1/2} e^{-z} [1+ O(1/z)],
\qquad z \to \infty,
\end{equation}
\begin{equation}\label{as2}
K_\mu(z) = O( z^{-|{\rm Re}\mu|}), \ z \to 0, \ \mu \ne 0,
\end{equation}
\begin{equation}\label{as3}
K_0(z) = -\log z + O(1), \ z \to 0.
\end{equation}
When $|\tau| \to \infty$ and $x > 0$ is fixed it behaves as
\begin{equation}\label{as4}
K_{i\tau}(x)= O\left({e^{-\pi|\tau|/2}\over \sqrt {|\tau|}}\right).
\end{equation}
We will  appeal in the sequel to the uniform inequality for the
modified Bessel function
\begin{equation}\label{inequn}
\left|K_{i\tau}(x)\right| \le {x^{-1/4}\over \sqrt{ \sinh \pi\tau
}}, \ x, \tau >0
\end{equation}
and its representation via the following Fourier cosine
integral
\begin{equation}\label{coskl}
\cosh\left({\pi\tau \over 2}\right) K_{i\tau}(x)= \int_0^\infty \cos
\tau u \cos(x \sinh u) du, \ x >0.
\end{equation}
Furthermore, employing relation (2.16.48.20) in \cite{Prud} and
making differentiation by a parameter we derive useful integral with
respect to an index of the modified Bessel function
\begin{equation}\label{intind}
\begin{split}
& {1\over \pi}\int_{-\infty}^\infty \tau e^{\lambda \tau} \
\left(t+ (1+t^2)^{1/2}\right)^{i\tau} K_{i\tau}(x) d\tau\\
&\hspace{1cm} = x \exp\left(-x \left[(1+t^2)^{1/2}\cos \lambda - i t
\sin\lambda\right]\right) \left[(1+t^2)^{1/2}\sin \lambda + it \cos
\lambda\right], \quad  x,t  >0
\end{split}
\end{equation}
where $0 \le \lambda < {\pi \over 2}$.

\section{Some Rajchman measures}

In this section we prove several  theorems, characterizing Rajchman
measures, which are associated with Fourier-Stieltjes integrals over
finite and infinite intervals.

We begin with  the following general result.

\begin{thm}\label{gensal} Let  $\varphi$ be a real- valued continuous integrable function of
bounded variation on $(0,\infty)$ vanishing at infinity. Then
$\varphi$ supports a Rajchman measure relatively its
Fourier-Stieltjes transform
\begin{equation}\label{fourraj}
\Phi(t)= \int_{0}^\infty e^{ixt}\  d\varphi(x),
\end{equation}
if and only if it has a limit at infinity $(|t| \to \infty)$.
\end{thm}
\begin{proof} Without loss of generality we prove the theorem for positive
$t$. Evidently, the necessity is trivial and we will prove the
sufficiency. Suppose that the limit of $\Phi(t)$ when $t \to
+\infty$ exists. Since $\Phi(t)= \Phi_c(t) + i \Phi_s(t)$, where
\begin{equation}\label{cosraj}
\Phi_c(t)= \int_{0}^\infty \cos xt\  d\varphi(x),
\end{equation}
\begin{equation}\label{sinraj}
\Phi_s(t)= \int_{0}^\infty \sin xt\  d\varphi(x),
\end{equation}
we will treat  these transforms separately. Taking (\ref{cosraj})
and integrating by parts we get
\begin{equation}\label{phicos}
\Phi_c(t)= - \varphi(0) + t \int_{0}^\infty \varphi(x) \sin xt\ dx.
\end{equation}
However, since $\varphi \in L_1\left(\mathbb{R}_+\right)$, we appeal
to the integrated form of the Fourier formula (cf. \cite{Tit}, Th.
22) to write for all $x \ge 0$
\begin{equation*}
\int_0^x \varphi(y) \ dy= {2\over \pi} \int_{0}^\infty \frac{1- \cos
yx} {y} \int_{0}^\infty \varphi(u) \sin uy\ du.
\end{equation*}
But taking into account the previous equality after simple change of
variable we come out with the relation
\begin{equation*}
{1\over x} \int_0^x \varphi(y) \ dy= {2\over \pi} \int_{0}^\infty
\frac{1- \cos y} {y^2}\left[\varphi(0)+ \Phi_c\left({y\over
x}\right)\right]dy, \ x >0.
\end{equation*}
Minding the value of elementary Feijer type integral
\begin{equation*}
{2\over \pi} \int_{0}^\infty \frac{1- \cos y} {y^2} dy=1,
\end{equation*}
we establish an important equality
\begin{equation}\label{eq1}
{1\over x} \int_0^x [\varphi(y)- \varphi(0)] \ dy= {2\over \pi}
\int_{0}^\infty \Phi_c\left({y\over x}\right)\frac{1- \cos y}
{y^2}dy, \ x >0.
\end{equation}
 Meanwhile, the left-hand side of (\ref{eq1}) is evidently goes to zero when $x \to 0+$ via the
continuity of $\varphi$ on $[0,\infty)$. Further, since $\varphi$ is
of bounded variation on $(0,\infty)$ we obtain the uniform estimate
\begin{equation*}
|\Phi_c(t)| \le \int_{0}^\infty \  d V_\varphi(x) = \Phi_0,
\end{equation*}
where $V_\varphi(x)$ is a variation of $\varphi$ on $[0,x]$ and
$\Phi_0 >0$ is a total variation of $\varphi$. This means that
$\Phi_c(t)$ is continuous and bounded on $\mathbb{R}_+$.
Furthermore, the integral with respect to $x$ in the right-hand side
of (\ref{eq1}) converges absolutely and uniformly by virtue of the
Weierstrass test. Consequently, since $\Phi_c(t)$ has a limit at
infinity, which is finite, say $a$, one can pass to the limit
through equality (\ref{eq1}) when $x \to 0+$. Hence we find
\begin{equation*}
\lim_{x\to 0+} {1\over x} \int_0^x [\varphi(y)- \varphi(0)] \ dy= {2
a\over \pi} \int_{0}^\infty \frac{1- \cos y} {y^2}dy= a = 0.
\end{equation*}
In order to complete the proof, we need  to verify whether the
Fourier sine transform (\ref{sinraj}) tends  to zero as well. To do
this, we appeal to the corresponding integrated form of the Fourier
formula for the Fourier cosine transform
\begin{equation}\label{Fform}
- \int_0^x \varphi(y) \ dy=  {2\over \pi} \int_{0}^\infty \frac{\sin
yx} {y^2} \Phi_s(y)\ dy, \ x >0,
\end{equation}
where after integration by parts $\Phi_s(t)$ turns to be represented
as follows
\begin{equation}\label{phis}
\Phi_s(t) = - t \int_{0}^\infty \varphi(u) \cos ut\ du, \ t >0.
\end{equation}
Hence it is easily seen that $\Phi_s(t)= O(t), \ t \to 0+$ and since
$|\Phi_s(t)| \le \Phi_0$ we have that ${\Phi_s(t)\over t} \in
L_2(\mathbb{R}_+)$. This means that the integral in the right-hand
side of (\ref{Fform}) converges absolutely and uniformly by $x \ge
0$. After simple change of variable we split the integral in the
right-hand side of  (\ref{Fform}) on two integrals to obtain
\begin{equation*}
- {1\over x }\int_0^x \varphi(y) \ dy= {2\over \pi } \int_{0}^1
\frac{\sin y} {y^2} \Phi_s\left({y\over x}\right)\ dy  + {2\over \pi
} \int_{1}^\infty \frac{\sin y} {y^2} \Phi_s\left({y\over x}\right)\
dy.
\end{equation*}
Considering again $x >0$ sufficiently small and splitting the
integral over $(0,1)$ on two more integrals over $(0,
x\log^\gamma(1/x))$ and $( x\log^\gamma(1/x), 1),$ where $0 < \gamma
< 1$,  we derive the equality
\begin{equation*}
 {2\over \pi } \int_{x\log^\gamma(1/x)}^1 \frac{\sin y} {y^2}
\Phi_s\left({y\over x}\right)\ dy = - {1\over x }\int_0^x \varphi(y)
\ dy - {2\over \pi } \int_{1}^\infty \frac{\sin y} {y^2}
\Phi_s\left({y\over x}\right)\ dy
\end{equation*}
\begin{equation*}
- {2\over \pi } \int_{0}^{x\log^\gamma(1/x)} \frac{\sin y} {y^2}
\Phi_s\left({y\over x}\right)\ dy.
\end{equation*}
Minding the inequality (see (\ref{phis})) $|\Phi_s(t)| \le t
||\varphi||_{L_1(\mathbb{R}_+)}, \ t \ge 0$, the right-hand side of
the latter equality has the straightforward estimate
\begin{equation}\label{estss}
\begin{split}
& \left|{1\over x }\int_0^x \varphi(y) \ dy + {2\over \pi }
\int_{1}^\infty \frac{\sin y} {y^2} \Phi_s\left({y\over x}\right)\
dy\right.\\
&  \left. + {2\over \pi } \int_{0}^{x\log^\gamma(1/x)} \frac{\sin y}
{y^2} \Phi_s\left({y\over x}\right)\ dy\right| \le \sup_{y \ge
0}|\varphi(y)| + {2\over \pi }\left[\Phi_0 +
||\varphi||_{L_1(\mathbb{R}_+)}\log^\gamma(1/x)\right].
\end{split}
\end{equation}
 On the other hand, via the first mean value theorem
\begin{equation*}
 {2\over \pi } \left|\int_{x\log^\gamma(1/x)}^1 \frac{\sin y} {y^2}
\Phi_s\left({y\over x}\right)\ dy\right| ={2\over \pi
}|\Phi_s(\xi(x))| \int_{\log^\gamma(1/x)}^1 \frac{\sin y} {y^2}\ dy,
\end{equation*}
where
$$\log^\gamma\left({1\over x}\right) \le \xi(x) \le {1\over x}.$$
Meanwhile, we have
\begin{equation*} {2\over \pi
}\int_{x\log^\gamma(1/x)}^1 \frac{\sin y} {y^2}\ dy
> {2\sin 1 \over \pi }\int_{\log^\gamma(1/x)}^1 {dy\over
y}= {2\sin 1\over \pi}\log\left({1\over x\log^\gamma(1/x)}\right).
\end{equation*}
Consequently, combining with (\ref{estss}) we find
\begin{equation}\label{ests}
\begin{split}
& |\Phi_s(\xi(x))| <  {1\over \sin 1}\left[{\pi\over 2 }\sup_{y \ge
0}|\varphi(y)| + \Phi_0 +
||\varphi||_{L_1(\mathbb{R}_+)}\log^\gamma(1/x)\right]\log^{-1}\left({1\over x\log^\gamma(1/x)}\right)\\
& = o(1), \  x \to 0+.
\end{split}
\end{equation}
Thus making $x \to 0+$ we get $\xi(x) \to +\infty$ and therefore
there is a subsequence $t_n= \xi(x_n) \to \infty$ such that
$\lim_{n\to +\infty}|\Phi_s(t_n)|=0$. But since the limit of
$\Phi_s(t)$ exists, when $t \to +\infty$ it will be zero. So
$\varphi$ supports a Raijcman measure and the theorem is proved.
\end{proof}

\begin{cor}\label{gensal1} Under conditions of Theorem \ref{gensal} $\varphi$ supports a Rajchman
measure if and only if two limits
\begin{equation*}
\lim_{t \to +\infty} \  t \int_{0}^\infty \varphi(x) \sin xt\ dx,  \quad \quad
\lim_{t \to +\infty} \  t \int_{0}^\infty \varphi(x) \cos xt\ dx
\end{equation*}
exist simultaneously (if so,  they equal to $\varphi(0)$ and $0$, respectively).
\end{cor}
More general result deals with the smoothness of the
Fourier-Stieltjes transform and a behavior at infinity of its
derivatives.

We have

\begin{cor}\label{gensal2} Let $n \in \mathbb{N}_0, \ \varphi(x), \ x \ge 0$ be a real-valued continuous function such
that $x^{m}\varphi(x)$ is of bounded variation on $[0, \infty)$ for
each $ m =0,1,\dots, n$. If $\varphi(x)= o(x^{-n}), x \to \infty$
and $x^{n}\varphi(x) \in L_1(\mathbb{R}_+)$, then the corresponding
Fourier-Stieltjes transform $(\ref{fourraj})$ $\Phi(t)$ is $n$ times
differentiable on $\mathbb{R}_+$, its $n$-th order derivative  is
equal to
\begin{equation}\label{der}
\Phi^{(n)}(t)= \int_{0}^\infty (ix)^{n} e^{itx} \ d\varphi(x)
\end{equation}
and vanishes at infinity if and only if there exists a limit of the
integral
\begin{equation*}
\Psi_n(t)= \int_{0}^\infty  e^{itx} \ d\left( x^n \varphi(x)\right)
\end{equation*}
when $|t| \to \infty.$
\end{cor}

\begin{proof} In fact, under conditions of the corollary  one can
differentiate $n$ times under the integral sign in the
Fourier-Stieltjes transform (\ref{fourraj}) via the absolute and
uniform convergence. Precisely, this circumstance is guaranteed  by
the estimate
\begin{equation*}
\left|\int_{0}^\infty (ix)^{m} e^{itx}\  d \varphi(x)\right|=
\left|\int_{0}^\infty e^{itx} \ d \left((ix)^{m}
\varphi(x)\right)\right.
\end{equation*}
\begin{equation*}
\left. - m \ i^m \int_{0}^\infty x^{m-1} \varphi(x) e^{itx} \
dx\right| \le Var_{[0,\infty)}  \left(x^{m} \varphi(x)\right)
\end{equation*}
\begin{equation*}
 + m \int_{0}^\infty x^{m-1} |\varphi(x)| \ dx = \Phi_m < \infty,\ m
 = 0,1,\dots,n,
\end{equation*}
where the latter integral  is finite since $x^{n}\psi(x) \in
L_1(\mathbb{R}_+)$ and $\varphi$ is continuous. Thus (\ref{der})
holds and in order to complete the proof we write it as
\begin{equation*}
\Phi^{(n)}(t)= i^n \left[ \int_{0}^\infty e^{itx} \ d\left(x^{n}
\varphi(x)\right) - n \int_{0}^\infty x^{n-1} \varphi(x) e^{itx}
dx\right].
\end{equation*}
The second integral of this equality tends to zero when $t \to
\infty$ via the Riemann-Lebesgue lemma. Therefore,
$\Phi^{(n)}(t)=o(1), \ t \to \infty, \ n \in \mathbb{N}_0$ if and
only if the first integral has a limit at infinity and this limit is
certainly zero.
\end{proof}

To finish a characterization of such a kind of Rajchman measures we
will prove one more theorem involving  general finite Fourier and
Fourier-Stieltjes transforms of a continuous function of bounded
variation on $[0,1]$
\begin{equation}\label{psif}
\hat{\psi}(t)= \int_0^1 e^{itx} \psi(x) dx,
\end{equation}
\begin{equation}\label{psifs}
\Psi(t)= \int_0^1 e^{itx}\  d \psi(x).
\end{equation}

\begin{thm}\label{gensal3} Let $\psi(x)$ be a continuous function  of bounded variation on $[0, 1]$
such that $\psi(0)=\psi(1)=0$ and $\psi(x)/ x \in L_2(0,1)$. If a
derivative of  Fourier transform (\ref{psif}) of $\psi$ satisfies
conditions:
\begin{equation}\label{klr}
i)\quad t^m \hat{\psi}^{(m)}(t) \in L_1(1,\infty), \ m= 0,1,2, \quad
ii)\quad  t^m \hat{\psi}^{(m-1)}(t)=o(1), \ t \to \infty, \ m=1,2,
\end{equation}
then $\psi$ is a Rajchman measure, i.e. $\Psi(t)=o(1), \ |t| \to
\infty.$
\end{thm}

\begin{proof} Without loss of generality we prove the theorem for positive $t$.
Taking (\ref{psifs}) we integrate by parts and eliminating
integrated terms come out with the equality
\begin{equation*}
\Psi(t)= -it \int_0^1 e^{itx} \psi(x) dx.
\end{equation*}
Meanwhile, passing to the limit through equality \eqref{intind}
when $\lambda \to {\pi\over 2}-$, we find
\begin{equation}\label{intind1}
\begin{split}
& {1\over \pi}\lim_{\lambda \to {\pi\over 2}-}
 \int_{-\infty}^\infty \tau e^{\lambda \tau} \
\left(t+ (1+t^2)^{1/2}\right)^{i\tau} K_{i\tau}(x) d\tau\\
&\hspace{1cm} = x (1+t^2)^{1/2} \ e^{ixt}, \quad  x,t  >0.
\end{split}
\end{equation}
Hence
\begin{equation}\label{l1}
\begin{split}
& \Psi(t)= {t \over \pi i (1+t^2)^{1/2}} \int\limits_{0}^{1} \psi(x) \\
&\hspace{1cm}\times  \lim_{\lambda \to {\pi\over 2}-}
 \int_{-\infty}^\infty \tau e^{\lambda \tau} \
\left(t+ (1+t^2)^{1/2}\right)^{i\tau} K_{i\tau}(x) d\tau\ {dx\over
x}.
\end{split}
\end{equation}
But since for each $x, t > 0$ and $0 \le \lambda < {\pi\over 2}$
(see \eqref{intind})
\begin{equation*}
\begin{split}
\left|\int_{-\infty}^\infty \tau e^{\lambda \tau} \ \left(t+
(1+t^2)^{1/2}\right)^{i\tau} K_{i\tau}(x) d\tau \right| \le x
\left[t+ (1+t^2)^{1/2}\right]
\end{split}
\end{equation*}
and $\psi$ is integrable we can take out the limit in \eqref{l1}
having the representation
\begin{equation}\label{l2}
\Psi(t)=   {t \over  \pi i (1+t^2)^{1/2}} \lim_{\lambda \to
{\pi\over 2}-} \int\limits_{0}^{1} \psi(x)\int_{-\infty}^\infty \tau
e^{\lambda \tau} \ \left(t+ (1+t^2)^{1/2}\right)^{i\tau}
K_{i\tau}(x) d\tau \ {dx\over x}.
\end{equation}
A change of the order of integration in \eqref{l2} is allowed via
Fubini's  theorem and can be easily justified employing inequality
(\ref{inequn}) and integrability of the function $\psi(x) x^{-5/4}$
over $(0,1)$, which is guaranteed by the condition $\psi(x)/x \in
L_2(0,1)$. Consequently,
\begin{equation}\label{raj}
\begin{split}
& \Psi(t)=  {t \over  \pi i (1+t^2)^{1/2}} \lim_{\lambda \to
{\pi\over 2}-}  \int_{-\infty}^\infty \tau e^{\lambda \tau} \
\left(t+ (1+t^2)^{1/2}\right)^{i\tau} \\
&\hspace{1cm} \times \int\limits_{0}^{1} K_{i\tau}(x) \psi(x)\
{dx\over x}\  d\tau.
\end{split}
\end{equation}
However, the inner integral with respect to $x$ in the latter
equality can be treated invoking with the Parseval equality for the
Fourier cosine transform. In fact,  calling representation
(\ref{coskl}) and asymptotic behavior of the modified Bessel
function (\ref{as1}), (\ref{as2}) we write
\begin{equation*}
\int\limits_{0}^{1} K_{i\tau}(x) \psi(x)\ {dx\over x} ={1\over
\cosh\left({\pi\tau \over 2}\right)} \int_0^\infty \cos \tau u
\int_0^1 \cos(x \sinh u) \psi(x)\ {dx du\over x}.
\end{equation*}
In the meantime, integrating by parts in the outer integral by $u$
when $|\tau|$ is big and taking into account that the integral by
$x$ vanishes when $u \to \infty$ due to the Riemann-Lebesgue lemma,
we obtain
\begin{equation}\label{part}
\begin{split}
& \int_0^\infty \cos \tau u \int_0^1 \cos(x \sinh u) \psi(x)\ {dx
du\over x}\\
&= {1\over \tau} \int_0^\infty \cosh u \ \sin \tau u \int_0^1 \sin
(x \sinh u) \psi(x)\ dxdu.
\end{split}
\end{equation}
Moreover, the latter integral converges absolutely and uniformly by
$|\tau|> A >0$ owing to the estimate
\begin{equation*}
\int_0^\infty \cosh u \ \left|\sin \tau u \int_0^1 \sin (x \sinh u)
\psi(x)\ dx\right| du \le \int_0^\infty \left|\int_0^1 \sin (x
y)\psi(x)\ dx\right| dy < \infty
\end{equation*}
and conditions (\ref{klr}) of the theorem. Integrating by parts two
more times in the right-hand side of (\ref{part}) we appeal to the
same conditions to derive the asymptotic relation
\begin{equation*}
\int\limits_{0}^{1} K_{i\tau}(x) \psi(x)\ {dx\over x}=
O\left({e^{-{\pi\over 2}|\tau|}\over \tau^3}\right), \ |\tau| \to
\infty.
\end{equation*}
Therefore one can pass to the limit by $\lambda$ in (\ref{raj}) and
then we observe, that $\Psi(t)$ plainly goes to zero when $t \to
+\infty$ owing to the Riemann-Lebesgue lemma, completing the proof
of the theorem.
\end{proof}

\section{An equivalent Salem's problem}

In this section we will formulate a problem, which is equivalent to  Salem's question [6],
having

\begin{cor}\label{sal}  The Fourier-Stieltjes transform
\begin{equation*}
 f(t)= \int_{0}^1 e^{ixt} d ?(x)
\end{equation*}
of the Minkowski question mark function vanishes at infinity, i.e. an answer on Salem's question is affirmative,
 if and only if two limit equalities
\begin{equation*}
\lim_{t \to +\infty} \  t \int_{0}^\infty ?\left({1\over x}\right)
\sin xt\ dx = 2,
\end{equation*}
\begin{equation*}
\lim_{t \to +\infty} \  t \int_{0}^\infty ?\left({1\over x}\right)
\cos xt\ dx = 0
\end{equation*}
take place simultaneously.
\end{cor}
\begin{proof} It follows  immediately  from double inequality (\ref{inF}),  simple equality due to functional equation $(\ref{rel3})$
\begin{eqnarray*}
\int\limits_{0}^{\infty}e^{ixt}\ d ?(x)= -
\int\limits_{0}^{\infty}e^{ixt}\ d ?\left({1\over x}\right)
\end{eqnarray*}
 and Corollary \ref{gensal1},  where we put  $\varphi(x)= ?(1/x), \ x >0, \ \varphi(0)= 2$.
\end{proof}

Finally, we generalize Salem's problem, proving

\begin{thm}\label{sal2} Let $k \in \mathbb{N}_0.$  If an answer on Salem's question is affirmative, then
\begin{equation}\label{deriv}
f^{(k)}(t)= \int_{0}^1 (ix)^{k} e^{itx} \ d ?(x)= o(1), \ |t| \to
\infty.
\end{equation}
\end{thm}

\begin{proof} It is easily seen that the Fourier- Stieltjes
transform of the Minkowski question mark function over $(0,1)$ is
infinitely differentiable and so for any $k \in \mathbb{N}_0$ we
have (\ref{deriv}). Suppose that $f^{(k)}$ does not tend to zero as
$|t| \to \infty$. Then we can find a sequence $\{t_m\}_{m=1}^\infty,
\ |t_m| \to \infty$ such that
\begin{equation*}
\left| \int_{0}^1 x^{k} e^{it_m x} \ d ?(x)\right| \ge \delta > 0.
\end{equation*}
Let ${t_m\over 2\pi}= n_m +\beta_m$, where $n_m$ is an integer and
$0 \le \beta_m < 1$. One can suppose that $\beta_m$ tends to a limit
$\beta$, we can always do it choosing again  subsequence from
$\{t_m\}$ if necessary. Hence
\begin{equation*}
|f^{(k)}(t_m)|=\left| \int_{0}^1 e^{2\pi i \beta x}x^{k} e^{2\pi in_m
x} \ d ?(x)\right| \ge \delta > 0.
\end{equation*}
But this contradicts to Salem's lemma \cite{Sal3}, p. 38, because
$f(2\pi n) \to 0, \ n \to \infty$ via assumption of the theorem  and the
Riemann -Stieltjes integral
\begin{equation*}
 \int_{0}^1 e^{2\pi i \beta x} x^{k} \ d ?(x)
\end{equation*}
converges for any $k \in \mathbb{N}_0$.
\end{proof}

\vspace{1cm } \noindent {\bf Acknowledgement}

\vspace{5mm}

\noindent The present investigation was supported  by the European
Regional Development Fund through the programme COMPETE and by the
Portuguese Government through the "Funda\cao\  para a Ci\^{e}ncia e
a Tecnologia" under the project PEst- C/MAT/UI0144/2011.

\vspace{2cm}

\noindent {\sc Semyon Yakubovich}, Department of Mathematics, Faculty of Sciences, University of Porto,
Campo Alegre st., 687,  4169-007 Porto, Portugal. {\tt syakubov@fc.up.pt}\\

\end{document}